\documentclass[a4paper,12pt]{article}
\usepackage{mathrsfs}
\usepackage{fancyhdr,graphicx}
\usepackage{epic}
\usepackage{amsfonts}
\usepackage{amssymb}
\usepackage{amsthm}
\usepackage{newlfont}
\usepackage{amsmath}
\usepackage{color}
\usepackage{ebezier}
\numberwithin{equation}{section}
\usepackage[top=2.5cm,bottom=2.5cm,left=2.5cm,right=2.5cm]{geometry}

\newtheorem{thm}{Theorem}[section]
\newtheorem{lemma}[thm]{Lemma}

\newtheorem{pro}[thm]{Proposition}

\usepackage{latexsym,bm}
\title{\vspace{0cm}{The Bondage Number of the Strong Product of a Complete Graph and a Path\thanks{This
work is supported by NSFC (grant no. 61073046).}}}

\author{Weisheng Zhao,\,Heping Zhang
\\\small{School of Mathematics and Statistics, Lanzhou
University, Lanzhou, Gansu 730000, P. R. China}
\\\small{E-mail addresses: weishengzhao101@aliyun.com, zhanghp@lzu.edu.cn
}}

\date{}

\begin{document}

\maketitle

\begin{abstract}
The bondage number $b(G)$ of a graph $G$ is the cardinality of a minimum edge set whose removal from $G$ results in a graph with the domination number greater than that of $G$. It is a parameter to measure the vulnerability of a communication network under link failure. In this paper, we obtain the exact value of the bondage number of the strong product of a complete graph and a path. That is, for any two integers $m\geq1$ and $n\geq2$, $b(K_{m}\boxtimes P_{n})=\lceil\frac{m}{2}\rceil$ if $n\equiv 0$ (mod 3); $m$ if $n\equiv 2$ (mod 3); $\lceil\frac{3m}{2}\rceil$ if $n\equiv 1$ (mod 3). Furthermore, we determine the exact value of the bondage number of the strong product of a complete graph and a special starlike tree.

\vspace{0.0cm}
\end{abstract}


%%introduction--------------------------------------------------
\section{Introduction}
All graphs considered in this paper are finite, undirected and simple. Let $G$ be a graph with vertex-set $V(G)$ and edge-set $E(G)$. A subset $D$ of $V(G)$ is called a \emph{dominating set} of $G$ if every vertex of $G$ is either in $D$ or adjacent to a vertex of $D$.  The \emph{domination number} $\gamma(G)$ is the cardinality of a minimum dominating set of $G$. A subset $S$ of $E(G)$ is called a \emph{bondage edge set} of a nonempty graph $G$ if $\gamma(G-S)>\gamma(G)$. The \emph{bondage number} of $G$, denoted by $b(G)$, is the cardinality of a minimum bondage edge set of $G$.

The  bondage number of a graph was coined by Fink, Jacobson, Kinch and Roberts \cite{Fink}. Before, it was called domination line-stability in \cite{Bauer}.

We know that a communication network can be modeled by a graph whose vertices represent fixed sites and whose edges represent communication links between these sites. The problem of finding a minimum dominating set in the graph corresponds to that of selecting a smallest set of sites at which to place transmitters such that every site in the network that does not have a transmitter is joined by a communication link to one that does have a transmitter. However, when some communication links malfunction, the transmitters may not transmit instructions to their neighboring sites normally. This corresponds to that when we remove some edges from a graph, some minimum dominating sets can not dominate all the vertices of the graph any longer.

Now, we consider such a question that what is the smallest number of edges we remove will render every minimum dominating set of the original graph to be a ``nondominating" set of the resulting graph. This smallest number is just the bondage number, which is a parameter to measure the vulnerability of a communication network under link failure.

In \cite{Hu1}, Hu and Xu showed that the problem of determining the bondage number of a general graph
is NP-hard. So it is significant to give out the value of the bondage number of a network. For some special graphs, the exact values of their bondage numbers have been obtained, such as complete graphs, paths, cycles, complete $t$-partition graphs \cite{Fink}, trees \cite{Fink,Hartnell,Teschner}, complete $t$-partite digraphs \cite{Zhang}, de Bruijn and Kautz digraphs \cite{Huang1}, etc.

Product graphs are one kind of important networks. On the study of the bondage number of the product graphs, so far, only the cartesian product graphs are considered. For example, $b(K_{n}\square K_{n})$ for $n\geq3$ \cite{Hartnell3,Teschner0}, $b(C_{n}\square P_{2})$ for $n\geq3$ \cite{Dunbar}, $b(C_{n}\square C_{3})$ for $n\geq4$ \cite{Sohn},   $b(C_{n}\square C_{4})$ for $n\geq4$ \cite{Kang}, $b(C_{n}\square C_{5})$ for $n\not\equiv3$ (mod 5) and $n\geq5$ \cite{Cao}, $b(P_{n}\square P_{2})$, $b(P_{n}\square P_{3})$ and $b(P_{n}\square P_{4})$ for $n\geq2$ \cite{Hu2} have been determined.

Let $G$ and $H$ be two graphs. The \emph{strong product} of $G$ and $H$, denote by $G \boxtimes H$, is a graph such that
 $V(G \boxtimes H)=V(G)\times V(H)$, two vertices $(g_{1}, h_{1})$ and $(g_{2}, h_{2})$ is adjacent
if and only if either $g_{1}=g_{2}$ and $h_{1}h_{2}\in E(H)$, or $g_{1}g_{2}\in E(G)$ and
$h_{1}=h_{2}$, or $g_{1}g_{2}\in E(G)$ and $h_{1}h_{2}\in E(H)$ (See \cite{Hammack}).
\begin{center}
\begin{picture}(328,76)
%the first graph
\multiput(0,0)(0,20){3}{\put(20,20){\line(1,0){60}}}
\multiput(0,0)(20,0){4}{\put(20,20){\line(0,1){40}}}

\multiput(0,0)(0,20){2}{\multiput(0,0)(20,0){3}{\put(0,0){\drawline(20,20)(40,40)}}}
\multiput(0,0)(0,20){2}{\multiput(0,0)(20,0){3}{\put(0,0){\drawline(20,40)(40,20)}}}

\multiput(0,0)(20,0){4}{\multiput(20,20)(0,20){3}{\circle*{2.5}}}

\put(12,20){\line(0,1){40}}
\multiput(12,20)(0,20){3}{\circle*{2.5}}
\put(20,12){\line(1,0){60}}
\multiput(20,12)(20,0){4}{\circle*{2.5}}

\put(29,-7){$P_{4}\boxtimes P_{3}$}

%the second graph
\multiput(0,0)(0,20){3}{\put(140,20){\line(1,0){60}}}
\multiput(0,0)(20,0){4}{\put(140,20){\line(0,1){20}}}

\multiput(0,0)(20,0){3}{\put(0,0){\drawline(140,20)(160,40)}}
\multiput(0,0)(20,0){3}{\put(0,0){\drawline(140,40)(160,20)}}

\multiput(0,0)(20,0){4}{\multiput(140,20)(0,20){3}{\circle*{2.5}}}

\put(140,12){\line(1,0){60}}
\multiput(140,12)(20,0){4}{\circle*{2.5}}
\put(132,20){\line(0,1){20}}
\multiput(132,20)(0,20){3}{\circle*{2.5}}

\put(127.5,-7){$P_{4}\boxtimes (K_{1}\cup K_{2})$}

%the third graph
\multiput(0,0)(0,20){2}{\put(260,40){\line(1,0){60}}}
\multiput(0,0)(20,0){4}{\put(260,40){\line(0,1){20}}}

\multiput(0,0)(20,0){3}{\put(0,0){\drawline(260,40)(280,60)}}
\multiput(0,0)(20,0){3}{\put(0,0){\drawline(260,60)(280,40)}}

\multiput(0,0)(20,0){2}{\put(0,0){\drawline(260,40)(300,60)}}
\multiput(0,0)(20,0){2}{\put(0,0){\drawline(260,60)(300,40)}}

\drawline(260,40)(320,60)
\drawline(260,60)(320,40)

\cbezier[500](260,40)(270,33)(290,33)(300,40)
\cbezier[500](280,40)(290,33)(310,33)(320,40)
\cbezier[500](260,40)(280,25)(300,25)(320,40)

\cbezier[500](260,60)(270,67)(290,67)(300,60)
\cbezier[500](280,60)(290,67)(310,67)(320,60)
\cbezier[500](260,60)(280,75)(300,75)(320,60)

\multiput(0,0)(20,0){4}{\multiput(260,40)(0,20){2}{\circle*{2.5}}}
\multiput(252,40)(0,20){2}{\circle*{2.5}}

\drawline(260,21)(320,21)
\cbezier[500](260,21)(270,14)(290,14)(300,21)
\cbezier[500](280,21)(290,14)(310,14)(320,21)
\cbezier[500](260,21)(280,6)(300,6)(320,21)
\multiput(260,21)(20,0){4}{\circle*{2.5}}

\put(252,40){\line(0,1){20}}

\put(268,-7){$K_{4}\boxtimes P_{2}$}

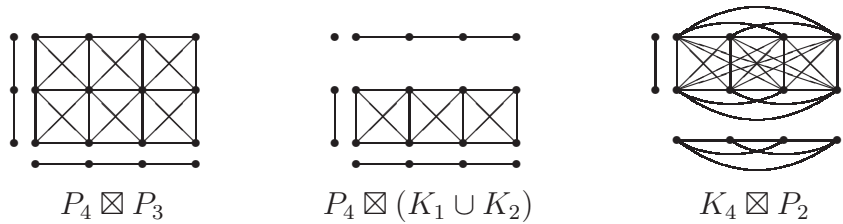
\put(67,-32){Figure 1. Examples of strong product.}
\end{picture}
\end{center}

\

\

 In this paper, we are going to study on strong product. We obtain the exact value of the bondage number of the strong product of a complete graph and a path. Furthermore, we determine the exact value of the bondage number of the strong product of a complete graph and a special starlike tree.

\section{Preliminary}
First, let us introduce some notations and terminologies. Denote by $N_{G}(v)$ and $N_{G}[v]$ the open and closed neighborhood of vertex $v$ in $G$, respectively. For any $\emptyset\neq X\subseteq V(G)$, we denote by $G[X]$ the subgraph of $G$ induced by $X$. For any $Y_{1},Y_{2}\subseteq V(G)$, let $[Y_{1},Y_{2}]$ denote the set of edges with one end in $Y_{1}$ and the other in $Y_{2}$. Denote by $\underline{MDS}(G)$ the set of all the minimum dominating sets of $G$. That is, $\underline{MDS}(G)=\{D \ | \ D$ is a minimum dominating set of $G\}$. For any two graphs $G$ and $H$, if $v\in V(H)$ and $xy\in E(H)$, then we wtite $G\boxtimes \{v\}= G\boxtimes H[\{v\}]$ and $G\boxtimes \{xy\}= G\boxtimes H[\{x,y\}]$ for short, respectively.

Next, we state some useful results below.
A set $S \subseteq V(G)$ is called a {\it $k$-packing} of graph $G$
if $d(x,y)>k$ for every pair of distinct vertices $x, y \in S$. The
{\it $k$-packing number} $P_k(G)$ is the cardinality of a
maximum $k$-packing of $G$. Let $K_{m}$ and $P_{n}$ denote a complete graph and a path of order $m$ and $n$, respectively.
\begin{pro}\label{Tp2} \emph{\cite{Meir}}
For any tree $T$, $P_{2}(T)=\gamma(T)$.  \qed
\end{pro}

\begin{pro}\label{strong-p2} \emph{\cite{Nowakowski}}
If $H$ is a graph with $P_{2}(H)=\gamma(H)$, then $\gamma(G\boxtimes H)=\gamma(G)\gamma(H)$ for any graph $G$. \qed
\end{pro}

\begin{pro}\label{p-d} \emph{\cite{Jacobson}} \
$\gamma(P_{n})=\lceil\frac{n}{3}\rceil$. \qed
\end{pro}

\begin{thm}\label{k-p-d}
$\gamma(K_{m}\boxtimes P_{n})=\gamma(K_{m})\gamma(P_{n})=\lceil\frac{n}{3}\rceil$.
\end{thm}
\begin{proof}
Since $\gamma(K_{m})=1$, the theorem follows immediately by Propositions \ref{Tp2} to \ref{p-d}.
\end{proof}

\begin{pro}\label{k-b} \emph{\cite{Fink}}
For $m\geq2$, $b(K_{m})=\lceil\frac{m}{2}\rceil$. \qed
\end{pro}

\begin{pro}\label{p-b} \emph{\cite{Fink}} For $n\geq2,$
\begin{equation*}
 \ \ \ \ \ \ \ \ \ \ \ \ \ \ \ \ \ \ \ \ \ \ \ \ \ \ \ \ \ \ \ \ \ \ \ \ \ \ b(P_{n})=
\begin{cases}
2, & \text{if $n\equiv 1 \ (\emph{mod} \ 3);$}\\
1, & \text{otherwise.} \ \ \ \ \ \ \ \ \ \ \ \ \ \ \ \ \ \ \ \ \ \ \ \ \ \ \ \ \ \ \ \ \ \ \ \ \ \ \ \ \qed
\end{cases}
\end{equation*}

\end{pro}

\section{Bondage number of $K_{m}\boxtimes P_{n}$}
In this section, we always let $V(K_{m})=\{u_1,u_2,\ldots,u_m\}$, $V(P_{n})=\{v_1,v_2,\ldots,v_n\}$ and $G=K_{m}\boxtimes P_{n}$. Set $R_{i}=V(K_{m}\boxtimes \{v_{i}\})=\{(u_1,v_i), (u_2,v_i), \ldots, (u_m,v_i)\}$, $i=1,2,\ldots,n$. For any $1\leq i\leq j\leq n$, write $B_{i}^{j}=G[\bigcup\limits_{l=i}^{j} R_{l}]$. If $j-i\geq2$, then we
let $E^{\ast}(B_{i}^{j})=E(B_{i}^{j})-(E(B_{i}^{i})\cup E(B_{j}^{j}))$.

\subsection{Some properties of a minimum dominating set of $K_{m}\boxtimes H$ and $K_{m}\boxtimes P_{n}$}
\begin{lemma}\label{block-1,2}
Let $H$ be a graph, $v\in V(H)$ and $xy\in E(H)$. Then $K_{m}\boxtimes \{v\}\cong K_{m}$ and $K_{m}\boxtimes \{xy\}\cong K_{2m}$.
\end{lemma}
\begin{proof}
It is immediate from the definition of strong product.
\end{proof}

\begin{lemma}\label{v}
Let $S\in\underline{MDS}(K_{m}\boxtimes H)$ and $v\in V(H)$. If $S\cap V(K_{m}\boxtimes \{v\})\neq\emptyset$ then $|S\cap V(K_{m}\boxtimes \{v\}|=1$.
\end{lemma}
\begin{proof}
Assume to the contrary that $|S\cap V(K_{m}\boxtimes \{v\})|\geq2$ and let $(a_{1},b_{1}),(a_{2},b_{2})\in S\cap V(K_{m}\boxtimes \{v\})$. Then $S-\{(a_{1},b_{1})\}$ is still a dominating set of $K_{m}\boxtimes H$, which contradicts to the minimality of $S$.
\end{proof}

\begin{lemma}\label{st}
Let $S\in\underline{MDS}(K_{m}\boxtimes H)$. If $s_{0}$ is a vertex of $H$ with degree one and $t_{0}$ is a neighbor of $s_{0}$ in $H$, then $|S\cap V(K_{m}\boxtimes \{s_{0}t_{0}\}|=1$.
\end{lemma}
\begin{proof}
Clearly, $|S\cap V(K_{m}\boxtimes\{s_{0}t_{0}\})|\geq1$. Otherwise, no element of $S$ can dominate $V(K_{m}\boxtimes\{s_{0}\})$ in $K_{m}\boxtimes H$.
Next, we need to prove that $|S\cap V(K_{m}\boxtimes\{s_{0}t_{0}\})|\leq1$. Assume to the contrary that $|S\cap V(K_{m}\boxtimes\{s_{0}t_{0}\})|\geq2$ and let $(a_{1},b_{1}),(a_{2},b_{2})\in S\cap V(K_{m}\boxtimes\{s_{0}t_{0}\})$. By Lemma \ref{v}, we can suppose without loss of generality that $(a_{1},b_{1})\in S\cap V(K_{m}\boxtimes\{s_{0}\})$ and $(a_{2},b_{2})\in S\cap V(K_{m}\boxtimes\{t_{0}\})$. But now, we have that $S-\{(a_{1},b_{1})\}$ is a dominating set of $K_{m}\boxtimes H$, which is contrary to the minimality of $S$.
\end{proof}

\begin{lemma}\label{block-d}
Let $D\in \underline{MDS}(G)$. Then $|D\cap V(B_{1}^{i})|\geq \gamma(B_{1}^{i-1})$ for every $1<i\leq n$ and $|D\cap V(B_{j}^{n})|\geq \gamma(B_{j+1}^{n})$ for every $1\leq j<n$.
\end{lemma}
\begin{proof}
By symmetry, we only prove that $|D\cap V(B_{1}^{i})|\geq \gamma(B_{1}^{i-1})$. Note that $B_{i}^{i}\cong K_{m}$ and $B_{1}^{i-1}\cong K_{m}\boxtimes P_{i-1}$. If $D\cap V(B_{i}^{i})\neq\emptyset$, then $D\cap V(B_{1}^{i})$ is a dominating set of $B_{1}^{i}$, and so $|D\cap V(B_{1}^{i})|\geq\gamma(B_{1}^{i})=\lceil\frac{i}{3}\rceil\geq\lceil\frac{i-1}{3}\rceil=\gamma(B_{1}^{i-1})$; if $D\cap V(B_{i}^{i})=\emptyset$, then $D\cap V(B_{1}^{i})=D\cap V(B_{1}^{i-1})$ is a dominating set of $B_{1}^{i-1}$, and so $|D\cap V(B_{1}^{i})|\geq\gamma(B_{1}^{i-1})$.
\end{proof}

\begin{lemma}\label{block-0}
Let $D\in \underline{MDS}(G)$.\\
(a) If $n\equiv0$ \emph{(mod 3),} then $|D\cap V(B_{i}^{i})|=0$ for every $1\leq i\leq n$ with $i\equiv0, 1$ \emph{(mod 3).}
(b) If $n\equiv2$ \emph{(mod 3),} then $|D\cap V(B_{i}^{i})|=0$ for every $1\leq i\leq n$ with $i\equiv0$ \emph{(mod 3).}
\end{lemma}
\begin{proof}
If $n\equiv0$ (mod 3) and $i=1$, then by Lemma \ref{block-d} and Theorem \ref{k-p-d}, we have $|D|= |D\cap V(B_{1}^{1})|+|D\cap V(B_{2}^{n})|\geq0+\gamma(B_{3}^{n})=\gamma(K_{m}\boxtimes P_{n-2})=\lceil\frac{n-2}{3}\rceil=\lceil\frac{n}{3}\rceil=|D|$, which implies that $|D\cap V(B_{1}^{1})|=0$. By symmetry, if $n\equiv0$ (mod 3) and $i=n$, we can get that $|D\cap V(B_{n}^{n})|=0$.

If $n\equiv0$ (mod 3) and $1< i< n$ with $i\equiv0, 1$ (mod 3), or $n\equiv2$ (mod 3) and $1\leq i\leq n$ with $i\equiv0$ (mod 3), then again
by Lemma \ref{block-d} and Theorem \ref{k-p-d}, we have $|D|= |D\cap V(B_{1}^{i-1})|+|D\cap V(B_{i}^{i})|+|D\cap V(B_{i+1}^{n})|\geq\gamma(B_{1}^{i-2})+0+\gamma(B_{i+2}^{n})=\lceil\frac{i-2}{3}\rceil+\lceil\frac{n-(i+1)}{3}\rceil=\lceil\frac{n}{3}\rceil=|D|$, which implies that $|D\cap V(B_{i}^{i})|=0$.
\end{proof}

\subsection{Upper bound of the bondage number of $K_{m}\boxtimes P_{n}$}
Let $H$ be a graph, $v\in V(H)$ and $xy\in E(H)$. We define two subsets $Z_{v}^{-}$ and $Z_{xy}^{\thinspace\shortmid}$ of $E(K_{m}\boxtimes \{v\})$ and  $E(K_{m}\boxtimes \{xy\})$ respectively as follows:
\begin{equation*}
Z_{v}^{-}=
\begin{cases}
\{(u_{1},v)(u_{2},v),(u_{3},v)(u_{4},v),\ldots,(u_{m-2},v)(u_{m-1},v)\}\\
\ \ \ \cup \ \{(u_{m-1},v)(u_{m},v)\},  &\text{if $m$ is odd and $m\geq3$;}\\
\{(u_{1},v)(u_{2},v),(u_{3},v)(u_{4},v),\ldots,(u_{m-1},v)(u_{m},v)\}, &\text{if $m$ is even.}
\end{cases}\\
\end{equation*}
\begin{center}
\ \ $Z_{xy}^{\thinspace\shortmid}=\{(u_{1},x)(u_{1},y),(u_{2},x)(u_{2},y),\ldots,(u_{m},x)(u_{m},y)\}$. \ \ \ \ \ \ \ \ \ \ \ \ \  \ \ \ \ \ \ \ \ \ \ \ \ \ \ \ \ \ \ \ \ \ \ \ \ \ \ \ \ \ \ \ \ \ \
\end{center}
\begin{lemma}\label{vxy}
Let $H$ be a graph, $v\in V(H)$ and $xy\in E(H)$. Then $Z_{v}^{-}$ and $Z_{xy}^{\thinspace\shortmid}$ are bondage edge sets of $K_{m}\boxtimes \{v\}$ and $K_{m}\boxtimes \{xy\}$, respectively.
\end{lemma}
\begin{proof}
Since $K_{m}\boxtimes \{v\}\cong K_{m}$ and $Z_{v}^{-}$ covers all the vertices of $K_{m}\boxtimes \{v\}$, we have $\gamma(K_{m}\boxtimes \{v\}-Z_{v}^{-})>1=\gamma(K_{m}\boxtimes \{v\})$. That is to say, $Z_{v}^{-}$ is a bondage edge set of $K_{m}\boxtimes \{v\}$.

Since $K_{m}\boxtimes \{xy\}\cong K_{2m}$ and $Z_{xy}^{\thinspace\shortmid}$ covers all the vertices of $K_{m}\boxtimes \{xy\}$, similarly, we can get that $Z_{xy}^{\thinspace\shortmid}$ is a bondage edge set of $K_{m}\boxtimes \{xy\}$.
\end{proof}

\begin{lemma}\label{upper-3}
If $m\geq2$ and $n\equiv 0$ \emph{(mod 3)}, then $b(K_{m}\boxtimes P_{n})\leq \lceil\frac{m}{2}\rceil$.
\end{lemma}

\begin{proof}
Since $|Z_{v_{2}}^{-}|=\lceil\frac{m}{2}\rceil$, it suffice to prove that $\gamma(G-Z_{v_{2}}^{-})>\gamma(G)$. Suppose to the contrary that $\gamma(G-Z_{v_{2}}^{-})\leq\gamma(G)$, which implies that $\gamma(G-Z_{v_{2}}^{-})=\gamma(G)$. So, for $D\in \underline{MDS}(G-Z_{v_{2}}^{-})$, we have $D\in \underline{MDS}(G)$.

 From Lemma \ref{block-0} (a), we obtain that $|D\cap V(B_{1}^{1})|=0=|D\cap V(B_{3}^{3})|$, which implies that $D\cap V(B_{2}^{2})\neq\emptyset$, and so we have $|D\cap V(B_{2}^{2})|=1$ by Lemma \ref{v}. Since $|D\cap V(B_{1}^{1})|=0=|D\cap V(B_{3}^{3})|$, it follows that $D\cap V(B_{2}^{2})$ is a dominating set of $B_{2}^{2}-Z_{v_{2}}^{-}$. Hence $\gamma(B_{2}^{2}-Z_{v_{2}}^{-})\leq |D\cap V(B_{2}^{2})|=1=\gamma(B_{2}^{2})$, which implies that $\gamma(B_{2}^{2}-Z_{v_{2}}^{-})=\gamma(B_{2}^{2})$.  But, it is impossible since $Z_{v_{2}}^{-}$ is a bondage edge set of $B_{2}^{2}$. Thus $\gamma(G-Z_{v_{2}}^{-})>\gamma(G)$. The lemma follows.
\end{proof}

\begin{lemma}\label{upper-2}
If $n\equiv 2$ \emph{(mod 3)}, then $b(K_{m}\boxtimes P_{n})\leq m$.
\end{lemma}

\begin{proof}
Since $|Z_{v_{1}v_{2}}^{\thinspace\shortmid}|=m$, it needs only to prove that $\gamma(G-Z_{v_{1}v_{2}}^{\thinspace\shortmid})>\gamma(G)$. Suppose to the contrary that $\gamma(G-Z_{v_{1}v_{2}}^{\thinspace\shortmid})=\gamma(G)$. Then, for $D\in \underline{MDS}(G-Z_{v_{1}v_{2}}^{\thinspace\shortmid})$, we have $D\in \underline{MDS}(G)$.

 By Lemma \ref{block-0} (b), we have $|D\cap V(B_{3}^{3})|=0$, which implies that $D\cap V(B_{1}^{2})$ is a dominating set of $B_{1}^{2}-Z_{v_{1}v_{2}}^{\thinspace\shortmid}$. By Lemma \ref{st}, we have $|D\cap V(B_{1}^{2})|=1$. Thus $\gamma(B_{1}^{2}-Z_{v_{1}v_{2}}^{\thinspace\shortmid})\leq|D\cap V(B_{1}^{2})|=1=\gamma(B_{1}^{2})$, and hence $\gamma(B_{1}^{2}-Z_{v_{1}v_{2}}^{\thinspace\shortmid})=\gamma(B_{1}^{2})$, which is contrary to that $Z_{v_{1}v_{2}}^{\thinspace\shortmid}$ is a bondage edge set of  $B_{1}^{2}$.
\end{proof}

\begin{lemma}\label{suspended vertex}
Let $H$ be a graph which contains at least one vertex of degree one. If $m\geq2$, then $b(K_{m}\boxtimes H)\leq \lceil\frac{3m}{2}\rceil$.
\end{lemma}
\begin{proof}
Let $s_{0}$ be a vertex of $H$ with degree one, $t_{0}$ be a neighbor of $s_{0}$ in $H$. Since $|Z_{s_{0}}^{-}\cup Z_{s_{0}t_{0}}^{\thinspace\shortmid}|=|Z_{s_{0}}^{-}|+|Z_{s_{0}t_{0}}^{\thinspace\shortmid}|=\lceil\frac{m}{2}\rceil+m=\lceil\frac{3m}{2}\rceil$,
 it suffices to prove that $\gamma(K_{m}\boxtimes H-(Z_{s_{0}}^{-}\cup Z_{s_{0}t_{0}}^{\thinspace\shortmid}))>\gamma(K_{m}\boxtimes H)$. Suppose to the contrary that $\gamma(K_{m}\boxtimes H-(Z_{s_{0}}^{-}\cup Z_{s_{0}t_{0}}^{\thinspace\shortmid}))=\gamma(K_{m}\boxtimes G)$.  Let $S\in \underline{MDS}(K_{m}\boxtimes H-(Z_{s_{0}}^{-}\cup Z_{s_{0}t_{0}}^{\thinspace\shortmid}))$, then we have $S\in \underline{MDS}(K_{m}\boxtimes H)$.

By Lemma \ref{st}, we can let $S\cap V(K_{m}\boxtimes\{s_{0}t_{0}\})=\{(u_{i_{0}},y_{0})\}$, where $1\leq i_{0}\leq m$ and $y_{0}\in\{s_{0},t_{0}\}$. Since $|S\cap V(K_{m}\boxtimes\{s_{0}t_{0}\})|=1$, it follows that all the vertices of $V(K_{m}\boxtimes\{s_{0}\})$ are only dominated by $(u_{i_{0}},y_{0})$ in $K_{m}\boxtimes H-(Z_{s_{0}}^{-}\cup Z_{s_{0}t_{0}}^{\thinspace\shortmid})$. However, if $y_{0}=s_{0}$, then $(u_{i_{0}-1},s_{0})$ or $(u_{i_{0}+1},s_{0})$ can not be dominated by $(u_{i_{0}},y_{0})$ in $K_{m}\boxtimes H-(Z_{s_{0}}^{-}\cup Z_{s_{0}t_{0}}^{\thinspace\shortmid})$, a contradiction; if $y_{0}=t_{0}$, then $(u_{i_{0}},s_{0})$ can not be dominated by $(u_{i_{0}},y_{0})$ in $K_{m}\boxtimes H-(Z_{s_{0}}^{-}\cup Z_{s_{0}t_{0}}^{\thinspace\shortmid})$, also a contradiction. Hence $\gamma(K_{m}\boxtimes H-(Z_{s_{0}}^{-}\cup Z_{s_{0}t_{0}}^{\thinspace\shortmid}))>\gamma(K_{m}\boxtimes H)$. The lemma follows.
\end{proof}

\begin{lemma}\label{upper-1}
If $m, n\geq2$, then $b(K_{m}\boxtimes P_{n})\leq \lceil\frac{3m}{2}\rceil$.
\end{lemma}
\begin{proof}
It is immediate from Lemma \ref{suspended vertex}.
\end{proof}

\subsection{Lower bound of the bondage number of $K_{m}\boxtimes P_{n}$}
A \emph{star} is a connected graph with at most one vertex of degree more than one, which is called the \emph{center} of the star. (If there is no vertex of degree
more than one, then any vertex can be the center.) It is easy to see that every star is isomorphic to a complete bipartite $K_{1,n}$ $(n\geq0)$.

\begin{lemma}\label{bondage star}
Let $c$ be the center of a star $K_{1,n}$, $Z\subseteq E(K_{m}\boxtimes K_{1,n})$ and $E^{\ast}(K_{m}\boxtimes K_{1,n})=E(K_{m}\boxtimes K_{1,n})-E(K_{m}\boxtimes K_{1,n}[N_{K_{1,n}}(c)])$. If $|Z\cap E^{\ast}(K_{m}\boxtimes K_{1,n})|<\lceil\frac{m}{2}\rceil$, then $\gamma(K_{m}\boxtimes K_{1,n}-Z)=\gamma(K_{m}\boxtimes K_{1,n})=1$.
\end{lemma}

\begin{proof}
Let $Z_{1}=Z\cap E^{\ast}(K_{m}\boxtimes K_{1,n})$, $Z_{2}=Z-Z_{1}$ and $X=V(K_{m}\boxtimes \{c\})=\{(k,c) \ | \ k\in V(K_{m})\}$. Since $|X|=m$ and $|Z_{1}|<\lceil\frac{m}{2}\rceil$, it follows that at least one vertex of $X$, say $(k_{0},c)$, is not covered by $Z_{1}$. Note that $\{(k,c)\}\in \underline{MDS}(K_{m}\boxtimes K_{1,n})$ for every $(k,c)\in X$. So we have $\{(k_{0},c)\}\in\underline{MDS}(K_{m}\boxtimes K_{1,n}-Z_{1})$.

Note that $Z_{2}\subseteq E(K_{m}\boxtimes K_{1,n}[N_{K_{1,n}}(c)])$. After removing $Z_{2}$ from $K_{m}\boxtimes K_{1,n}-Z_{1}$, $\{(k_{0},c)\}$ is still a dominating set of the resulting graph. That is to say, $\{(k_{0},c)\}\in\underline{MDS}(K_{m}\boxtimes K_{1,n}-Z_{1}-Z_{2})=\underline{MDS}(K_{m}\boxtimes K_{1,n}-Z)$. The lemma follows.
\end{proof}

\begin{lemma}\label{p3*-b}
Let $Z\subseteq E(K_{m}\boxtimes P_{n})$ and $E^{\ast}(B_{i}^{i+2})=E(B_{i}^{i+2})-(E(B_{i}^{i})\cup E(B_{i+2}^{i+2}))$. If
$|Z\cap E^{\ast}(B_{i}^{i+2})|<\lceil\frac{m}{2}\rceil$, then $\gamma(B_{i}^{i+2}-Z)=\gamma(B_{i}^{i+2})=1$.
\end{lemma}
\begin{proof}
Note that $B_{i}^{i+2}\cong K_{m}\boxtimes K_{1,2}$. By Lemma \ref{bondage star}, the lemma follows.
\end{proof}

\begin{lemma}\label{p2*-b}
Let $Z\subseteq E(K_{m}\boxtimes P_{n})$, and $E^{\ast}(B_{i}^{i+1})=E(B_{i}^{i+1})-E(B_{i}^{i})$ or $E(B_{i}^{i+1})-E(B_{i+1}^{i+1})$. If
$|Z\cap E^{\ast}(B_{i}^{i+1})|<\lceil\frac{m}{2}\rceil$, then $\gamma(B_{i}^{i+1}-Z)=\gamma(B_{i}^{i+1})=1$.
\end{lemma}
\begin{proof}
Note that $B_{i}^{i+1}\cong K_{m}\boxtimes K_{1,1}$. By Lemma \ref{bondage star}, the lemma follows.
\end{proof}

\begin{lemma}\label{p2-b}
$b(B_{i}^{i+1})= m$ for every $1\leq i< n$.
\end{lemma}

\begin{proof}
Since $B_{i}^{i+1}\cong K_{2m}$, the lemma follows from Proposition \ref{k-b}.
\end{proof}

\begin{lemma}\label{lower-3}
If $n\equiv 0$ \emph{(mod 3)}, then $b(K_{m}\boxtimes P_{n})\geq \lceil\frac{m}{2}\rceil$.
\end{lemma}

\begin{proof}
Let $Z\subseteq E(G)$ with $|Z|<\lceil\frac{m}{2}\rceil$. We need only to prove that $\gamma(G-Z)=\gamma(G)$. Let $D_{1},D_{2},\ldots,D_{\frac{n}{3}}$ be minimum dominating sets of subgraphs $B_{1}^{3}-Z,B_{4}^{6}-Z,\ldots,B_{n-2}^{n}-Z$, respectively.
By Lemma \ref{p3*-b}, we have $|D_{1}|=|D_{2}|=\cdots=|D_{\frac{n}{3}}|=1$. Thus, $\gamma(G-Z)\leq|\bigcup\limits_{i=1}^{n/3}D_{i}|=\sum\limits_{i=1}^{n/3}|D_{i}|=\frac{n}{3}=\gamma(G)$, which implies that $\gamma(G-Z)=\gamma(G)$.
\end{proof}

\begin{lemma}\label{lower-2}
If $n\equiv 2$ \emph{(mod 3)}, then $b(K_{m}\boxtimes P_{n})\geq m$.
\end{lemma}
\begin{proof}
Let $Z\subseteq E(G)$ with $|Z|<m$. It suffices to prove that $\gamma(G-Z)=\gamma(G)$.

If $|E(B_{i}^{i+2})\cap Z|< \lceil\frac{m}{2}\rceil$ for every $i\in\{1,4,\ldots,n-4\}$, let $D_{1},D_{2},\ldots,D_{\frac{n-2}{3}}$ and $D_{\lceil\frac{n}{3}\rceil}$ be the minimum dominating sets of $B_{1}^{3}-Z,B_{4}^{6}-Z,\ldots,B_{n-4}^{n-2}-Z$ and $B_{n-1}^{n}-Z$, respectively. By Lemmas \ref{p3*-b} and \ref{p2-b}, we have $|D_{1}|=|D_{2}|=\cdots=|D_{\frac{n-2}{3}}|=|D_{\lceil\frac{n}{3}\rceil}|=1$. Thus $\gamma(G-Z)\leq|\bigcup\limits_{i=1}^{\lceil n/3\rceil}D_{i}|=\lceil\frac{n}{3}\rceil=\gamma(G)$, which implies that $\gamma(G-Z)=\gamma(G)$.

If $|E(B_{i_{0}}^{i_{0}+2})\cap Z|\geq \lceil\frac{m}{2}\rceil$ for some $i_{0}\in\{1,4,\ldots,n-4\}$, let $J_{1},J_{2},\ldots,J_{\lceil\frac{n}{3}\rceil}$ be the minimum dominating sets of $B_{1}^{3}-Z,B_{4}^{6}-Z,\ldots,B_{i_{0}-3}^{i_{0}-1}-Z,B_{i_{0}}^{i_{0}+1}-Z,B_{i_{0}+2}^{i_{0}+4}-Z,B_{i_{0}+5}^{i_{0}+7}-Z,\ldots
,B_{n-2}^{n}-Z$, respectively.  We must have that $|E(B_{1}^{3})\cap Z|,|E(B_{4}^{6})\cap Z|,\ldots,|E(B_{i_{0}-3}^{i_{0}-1})\cap Z|,|E^{\ast}(B_{i_{0}+2}^{i_{0}+4})\cap Z|,|E(B_{i_{0}+5}^{i_{0}+7})\cap Z|,\ldots
,|E(B_{n-2}^{n})\cap Z|< m-\lceil\frac{m}{2}\rceil \leq \lceil\frac{m}{2}\rceil$. By
Lemmas \ref{p3*-b} and \ref{p2-b}, we have $|J_{1}|=|J_{2}|=\cdots=|J_{\lceil\frac{n}{3}\rceil}|=1$. Thus, we can get that $\gamma(G-Z)=|\bigcup\limits_{i=1}^{\lceil n/3\rceil}J_{i}|=\gamma(G)$.
\end{proof}

\begin{lemma}\label{lower-1-lemma}
Let $Z\subseteq E(K_{m}\boxtimes P_{n})$ with $|Z|<\lceil\frac{3m}{2}\rceil$. If $m,n\geq 2$, then there is a vertex of $V(B_{1}^{2})$ which can dominate $V(B_{1}^{1})$ in $K_{m}\boxtimes P_{n}-Z$.
\end{lemma}
\begin{proof}
If $|Z\cap E(B_{1}^{1})|<\lceil\frac{m}{2}\rceil$, then by Proposition \ref{k-b} and Lemma \ref{block-1,2}, we have $\gamma(B_{1}^{1}-Z)=\gamma(B_{1}^{1})=1$, and so the lemma is true.
If $|Z\cap E(B_{1}^{1})|\geq\lceil\frac{m}{2}\rceil$, then we have $|Z\cap [V(B_{1}^{1}),V(B_{2}^{2})]|<\lceil\frac{3m}{2}\rceil-\lceil\frac{m}{2}\rceil=m$. Thus, there is at least one vertex of $V(B_{2}^{2})$, say $(u_{r_{0}},v_{l_{0}})$, which can not be covered by $Z\cap [V(B_{1}^{1}),V(B_{2}^{2})]$. Now, it easy to see that $(u_{r_{0}},v_{l_{0}})$ is a vertex of $V(B_{1}^{2})$ which can dominate $V(B_{1}^{1})$ in $G-Z$.
\end{proof}

\begin{lemma}\label{lower-1}
If $m,n\geq 2$ and $n\equiv 1$ \emph{(mod 3)}, then $b(K_{m}\boxtimes P_{n})\geq \lceil\frac{3m}{2}\rceil$.
\end{lemma}
\begin{proof}
Let $Z\subseteq E(G)$ with $|Z|<\lceil\frac{3m}{2}\rceil$. It needs only to prove that $\gamma(G-Z)=\gamma(G)$.

\smallskip
{\it Case 1.} $|E(B_{i_{0}}^{i_{0}+1})\cap Z|\geq m$ for some $1\leq i_{0}< n$.

{\it Subcase 1.1.} $B_{i_{0}}^{i_{0}+1}=B_{1}^{2}$ or $B_{i_{0}}^{i_{0}+1}=B_{n-1}^{n}$.

If $B_{i_{0}}^{i_{0}+1}=B_{1}^{2}$, by Lemma \ref{lower-1-lemma}, there is a vertex of $V(B_{1}^{2})$, say $(u_{p_{0}},v_{q_{0}})$, which can dominate $V(B_{1}^{1})$ in $G-Z$.
Let $D_{1},D_{2},\ldots,D_{\frac{n-1}{3}}$ be minimum dominating sets of $B_{2}^{4}-Z,B_{5}^{7}-Z,\ldots,B_{n-2}^{n}-Z$, respectively.
By Lemma \ref{p3*-b}, we have $|D_{1}|=|D_{2}|=\cdots=|D_{\frac{n-1}{3}}|=1$. Thus $\gamma(G-Z)\leq|\bigcup\limits_{l=1}^{ (n-1)/3}D_{l}|+|\{(u_{p_{0}},v_{q_{0}})\}|\leq\lceil\frac{n}{3}\rceil=\gamma(G)$, which implies that $\gamma(G-Z)=\gamma(G)$.
Symmetrically, if $B_{i_{0}}^{i_{0}+1}=B_{n-1}^{n}$,  we can also get that $\gamma(G-Z)=\gamma(G)$.

\smallskip
{\it Subcase 1.2.} $B_{i_{0}}^{i_{0}+1}\neq B_{1}^{2}$ and $B_{i_{0}}^{i_{0}+1}\neq B_{n-1}^{n}$.

If $i_{0}\equiv 1$ (mod 3), let $L_{1},L_{2},\ldots,L_{\lceil\frac{n}{3}\rceil}$ be minimum dominating sets of $B_{1}^{2}-Z,B_{3}^{4}-Z,B_{5}^{7}-Z,B_{8}^{10}-Z,\ldots,B_{n-2}^{n}-Z$, respectively. We must have $|Z\cap E(B_{1}^{2})|$, $|Z\cap (E(B_{3}^{4})-E(B_{4}^{4}))|$, $|Z\cap E^{\ast}(B_{5}^{7})|, |Z\cap E^{\ast}(B_{8}^{10})|,\ldots,|Z\cap E^{\ast}(B_{n-2}^{n})|<\lceil\frac{3m}{2}\rceil-m=\lceil\frac{m}{2}\rceil$. By Lemmas \ref{p2*-b} and  \ref{p3*-b}, we have $|L_{1}|=|L_{2}|=\cdots=|L_{\lceil\frac{n}{3}\rceil}|=1$. And so
we can get that $\gamma(G-Z)=\sum\limits_{l=1}^{\lceil n/3\rceil}|L_{l}|=\lceil\frac{n}{3}\rceil=\gamma(G)$.

If $i_{0}\equiv 2$ (mod 3), we consider the minimum dominating sets of subgraphs $B_{1}^{2}-Z,B_{3}^{5}-Z,B_{6}^{8}-Z,\ldots,B_{n-4}^{n-2}-Z,B_{n-1}^{n}-Z$.
It is similar to the above paragraph, we can get that $\gamma(G-Z)=\gamma(G)$ by calculating the cardinality of the union of the minimum dominating sets of these subgraphs.

If $i_{0}\equiv 0$ (mod 3), we consider the minimum dominating sets of $B_{1}^{3}-Z,B_{4}^{6}-Z,\ldots,B_{n-6}^{n-4}-Z,B_{n-3}^{n-2}-Z,B_{n-1}^{n}-Z$. As a consequence, we have $\gamma(G-Z)=\gamma(G)$.

\smallskip
{\it Case 2.} $|E(B_{i}^{i+1})\cap Z|< m$ for every $1\leq i< n$.

{\it Subcase 2.1.} $|E(B_{j}^{j+2})\cap Z|< \lceil\frac{m}{2}\rceil$ for every $j\in\{1,4,\ldots,n-6\}$.

Let $J_{1},J_{2},\ldots,J_{\lceil\frac{n}{3}\rceil}$ be minimum dominating sets of $B_{1}^{3}-Z,B_{4}^{6}-Z,\ldots,B_{n-6}^{n-4}-Z,B_{n-3}^{n-2}-Z,B_{n-1}^{n}-Z$, respectively. (If $n=4$, we consider the minimum dominating sets of $B_{1}^{2}$ and $B_{3}^{4}$.)
By Lemmas \ref{p3*-b} and \ref{p2-b}, we have $|J_{1}|=|J_{2}|=\cdots=|J_{\lceil\frac{n}{3}\rceil}|=1$. Thus, we can get that $\gamma(G-Z)=\sum\limits_{l=1}^{\lceil n/3\rceil}|J_{l}|=\gamma(G)$.

\smallskip
{\it Subcase 2.2.} $|E(B_{j_{0}}^{j_{0}+2})\cap Z|\geq \lceil\frac{m}{2}\rceil$ for some $j_{0}\in\{1,4,\ldots,n-6\}$.

Assume without loss of generality that
\begin{center}
$j_{0}=$ min\large{$\{$}\normalsize{$j\in\{1,4,\ldots,n-6\} : |E(B_{j}^{j+2})\cap Z|\geq \lceil\frac{m}{2}\rceil$}\large{$\}$}\normalsize{.}
\end{center}

If $|E(B_{k}^{k+2})\cap Z|<\lceil\frac{m}{2}\rceil$ for every $k\in\{j_{0}+2,j_{0}+5,\ldots,n-4\}$, let $Q_{1},Q_{2},\ldots,Q_{\lceil\frac{n}{3}\rceil}$ be the minimum dominating sets of $B_{1}^{3}-Z,B_{4}^{6}-Z,\ldots,B_{j_{0}-3}^{j_{0}-1}-Z,B_{j_{0}}^{j_{0}+1}-Z,B_{j_{0}+2}^{j_{0}+4}-Z,B_{j_{0}+5}^{i_{0}+7}-Z,\ldots,B_{n-4}^{n-2}-Z
,B_{n-1}^{n}-Z$, respectively. By
Lemmas \ref{p3*-b} and \ref{p2-b}, we have $|Q_{1}|=|Q_{2}|=\cdots=|Q_{\lceil\frac{n}{3}\rceil}|=1$. Hence $\gamma(G-Z)=\bigcup\limits_{l=1}^{\lceil n/3\rceil}|Q_{l}|=\gamma(G)$.

If $|E(B_{k_{0}}^{k_{0}+2})\cap Z|\geq \lceil\frac{m}{2}\rceil$ for some $k_{0}\in\{j_{0}+2,j_{0}+5,\ldots,n-4\}$, let $W_{1},W_{2},\ldots,W_{\lceil\frac{n}{3}\rceil}$ be the minimum dominating sets of $B_{1}^{3}-Z,B_{4}^{6}-Z,\ldots,B_{j_{0}-3}^{j_{0}-1}-Z,B_{j_{0}}^{j_{0}+1}-Z,B_{j_{0}+2}^{j_{0}+4}-Z,B_{j_{0}+5}^{j_{0}+7}-Z,\ldots,
B_{k_{0}-3}^{k_{0}-1}-Z,B_{k_{0}}^{k_{0}+1}-Z,B_{k_{0}+2}^{k_{0}+4}-Z,B_{k_{0}+5}^{k_{0}+7}-Z,\ldots,B_{n-2}^{n}-Z
$, respectively. We must have $|Z\cap E(B_{1}^{3})|,|Z\cap E(B_{4}^{6})|,\ldots,|Z\cap E(B_{j_{0}-3}^{j_{0}-1})|,|Z\cap E^{\ast}(B_{j_{0}+2}^{j_{0}+4})|,|Z\cap E(B_{j_{0}+5}^{j_{0}+7})|,\ldots,
|Z\cap E(B_{k_{0}-3}^{k_{0}-1})|,|Z\cap E^{\ast}(B_{k_{0}+2}^{k_{0}+4})|,|Z\cap E(B_{k_{0}+5}^{k_{0}+7})|,\ldots,|Z\cap E(B_{n-2}^{n})|<\lceil\frac{3m}{2}\rceil-\lceil\frac{m}{2}\rceil-\lceil\frac{m}{2}\rceil\leq\lceil\frac{m}{2}\rceil
$. As a consequence, we have $\gamma(G-Z)=\sum\limits_{l=1}^{\lceil n/3\rceil}|W_{l}|=\gamma(G)$.
\end{proof}

\subsection{Exact value the bondage number of $K_{m}\boxtimes P_{n}$}
\begin{thm}\label{main-thm}
If $n\geq2$, then
\begin{equation*}
b(K_{m}\boxtimes P_{n})=
\begin{cases}
\lceil\frac{m}{2}\rceil, & \text{if $n\equiv0$ \emph{(mod 3);}}\\
m, & \text{if $n\equiv2$ \emph{(mod 3);}}\\
\lceil\frac{3m}{2}\rceil, & \text{if $n\equiv1$ \emph{(mod 3).}}
\end{cases}
\end{equation*}
\end{thm}
\begin{proof}
If $m\geq2$, then the theorem follows by Lemmas \ref{upper-3}, \ref{upper-2}, \ref{upper-1}, \ref{lower-3}, \ref{lower-2} and \ref{lower-1}; if $m=1$, then $K_{m}\boxtimes P_{n}\cong P_{n}$, and the theorem follows by Proposition \ref{p-b}.
\end{proof}
\section{A consequence on starlike tree}
A \emph{starlike tree} is a tree with at most one vertex of degree more than two, which is called the \emph{center} of the starlike tree. (If there is no vertex of degree
more than two, then any vertex can be the center.) We denote by $S(n_1,n_2,\ldots,n_l)$ a starlike
tree in which removing the center leaves disjoint paths $P_{n_1} ,P_{n_2},\ldots,P_{n_l}$ with orders $n_{1},n_{2},\ldots,n_{l}$ respectively, which are called \emph{branches} of the starlike tree. If let
$S=S(n_1,n_2,\ldots,n_l)$, $c$ be the center of $S$ and $\widetilde{P}_{n_i}=S[\{c\}\cup V(P_{n_i})]$, $i=1,2,\ldots,l$, then the subgraph $\widetilde{P}_{n_i}$ is called a \emph{augmented branch} of $S$.
In this section, we always set $P_{n_i}=x_{1}^{i}x_{2}^{i}\cdots x_{n_{i}}^{i}$ with $x_{1}^{i}$ being a neighbor of $c$ and $x_{n_{i}}^{i}$ being a vertex of degree one in $S$ , $i=1,2,\ldots,l$.

\begin{lemma}\label{sl-d-lemma}
Let $S=S(n_{1},n_{2},\ldots,n_{l})$ be a starlike tree, $c$ be the center of $S$, $D$ be a dominating set of $S$, $P_{n_{1}},P_{n_{2}},\ldots,P_{n_{l}}$ be the branches of $S$ and $\widetilde{P}_{n_{1}},\widetilde{P}_{n_{2}},\ldots,\widetilde{P}_{n_{l}}$ be the augmented branches corresponding to $P_{n_{1}},P_{n_{2}},\ldots,P_{n_{l}}$, respectively. Then, for $1\leq i \leq l$,

\noindent (a) if $n_{i}\equiv 1$ \emph{(mod 3)}, then $|D\cap V(P_{n_i})|\geq \lceil\frac{n_{i}}{3}\rceil-1$ and $|D\cap V(\widetilde{P}_{n_i})|\geq \lceil\frac{n_{i}}{3}\rceil;$

\noindent (b) if $n_{i}\equiv 2$ \emph{(mod 3)}, then $|D\cap V(P_{n_i})|\geq \lceil\frac{n_{i}}{3}\rceil;$

\noindent (c) if $n_{i}\equiv 0$ \emph{(mod 3)}, then $|D\cap (V(P_{n_i}-x_{1}^{i})|\geq \lceil\frac{n_{i}}{3}\rceil$.
\end{lemma}

\begin{proof}
(a) Let $n_{i}\equiv 1$ (mod 3). If $c\notin D$, then $D\cap V(P_{n_{i}})$ is a dominating set of $P_{n_{i}}$, and so $|D\cap V(\widetilde{P}_{n_{i}})|= |D\cap V(P_{n_{i}})|\geq \gamma(P_{n_{i}})=\lceil\frac{n_{i}}{3}\rceil$; if $c\in D$, then $D\cap V(\widetilde{P}_{n_{i}})$ is a dominating set of $\widetilde{P}_{n_{i}}$, and so $|D\cap V(\widetilde{P}_{n_{i}})|\geq \lceil\frac{n_{i}+1}{3}\rceil=\lceil\frac{n_{i}}{3}\rceil$, from which we get that $|D\cap V(P_{n_{i}})|=|D\cap V(\widetilde{P}_{n_{i}})|-1\geq \lceil\frac{n_{i}}{3}\rceil-1$.

(b) Let $n_{i}\equiv 2$ (mod 3).  If $x_{1}^{i}\in D$, then $D\cap V(P_{n_{i}})$ is a dominating set of $P_{n_{i}}$, and so $|D\cap V(P_{n_{i}})|\geq \gamma(P_{n_{i}})=\lceil\frac{n_{i}}{3}\rceil$; if $x_{1}^{i}\notin D$, then $D\cap V(P_{n_{i}})$ is a dominating set of $P_{n_{i}}-x_{1}^{i}$, and so $|D\cap V(P_{n_{i}})|\geq \gamma(P_{n_{i}}-x_{1}^{i})=\lceil\frac{n_{i}-1}{3}\rceil=\lceil\frac{n_{i}}{3}\rceil$.

(c) Let $n_{i}\equiv 0$ (mod 3). If $x_{2}^{i}\in D$, then $D\cap V(P_{n_{i}}-x_{1}^{i})$ is a dominating set of $P_{n_{i}}-x_{1}^{i}$, and so $|D\cap V(P_{n_{i}}-x_{1}^{i})|\geq \gamma(P_{n_{i}}-x_{1}^{i})=\lceil\frac{n_{i}-1}{3}\rceil=\lceil\frac{n_{i}}{3}\rceil$; if $x_{2}^{i}\notin D$, then $D\cap V(P_{n_{i}}-x_{1}^{i})$ is a dominating set of $P_{n_{i}}-x_{1}^{i}-x_{2}^{i}$, and so $|D\cap V(P_{n_{i}}-x_{1}^{i})|\geq \gamma(P_{n_{i}}-x_{1}^{i}-x_{2}^{i})=\lceil\frac{n_{i}-2}{3}\rceil=\lceil\frac{n_{i}}{3}\rceil$.
\end{proof}

\begin{thm}\label{sl-domi}
Let $S=S(n_{1},n_{2},\ldots,n_{r},n_{r+1},n_{r+2},\ldots,n_{r+s},n_{r+s+1},n_{r+s+2},\ldots,n_{r+s+t})$ be a starlike tree with $n_{1}\equiv n_{2}\equiv\dots\equiv n_{r}\equiv1$ \emph{(mod 3)}, $n_{r+1}\equiv n_{r+2}\equiv\cdots\equiv n_{r+s}\equiv2$ \emph{(mod 3)}, $n_{r+s+1}\equiv n_{r+s+2}\equiv\cdots\equiv n_{r+s+t}\equiv0$ \emph{(mod 3)} and $l=r+s+t$. Then
\begin{equation*}
\gamma(S)=
\begin{cases}
\sum\limits_{i=1}^{l}\lceil\frac{n_{i}}{3}\rceil-(r-1), & \text{if $r\geq1;$}\\
\sum\limits_{i=1}^{l}\lceil\frac{n_{i}}{3}\rceil, & \text{if $r=0$ and $s\geq1;$}\\
\sum\limits_{i=1}^{l}\lceil\frac{n_{i}}{3}\rceil+1, & \text{if $r=0$ and $s=0$.}
\end{cases}
\end{equation*}
\end{thm}

\begin{proof}
Define $P_{n_{1}},P_{n_{2}},\ldots,P_{n_{t}}$ and $\widetilde{P}_{n_{1}},\widetilde{P}_{n_{2}},\ldots,\widetilde{P}_{n_{t}}$ as in Lemma \ref{sl-d-lemma}.
Let $c$ be the center of $S$ and $D$ be a minimum dominating set of $S$.

By Lemma \ref{sl-d-lemma}, we have $|D \cap V(\widetilde{P}_{n_1})|\geq\lceil\frac{n_{1}}{3}\rceil$ if $r\geq1$ and $|D\cap V(P_{n_i})|\geq\lceil\frac{n_{i}}{3}\rceil-1$ for every $2\leq i\leq r$, $|D\cap V(P_{n_i})|\geq\lceil\frac{n_{i}}{3}\rceil$ for every $r+1\leq i\leq r+s$ and $|D\cap V(P_{n_i}-x_{1}^{i})|\geq\lceil\frac{n_{i}}{3}\rceil$ for every $r+s+1\leq i\leq l$. Note that $|D\cap N_{S}[c]|\geq 1$. So we obtain that
\begin{equation*}
\gamma(S)=|D|\geq
\begin{cases}
|D \cap V(\widetilde{P}_{n_1})|+\sum\limits_{i=2}^{r}|D\cap V(P_{n_i})|+\sum\limits_{i=r+1}^{r+s}|D\cap V(P_{n_i})|\\ \ \ \ \  +\sum\limits_{i=r+s+1}^{l}|D\cap V(P_{n_i}-x_{1}^{i})|\geq\sum\limits_{i=1}^{l}\lceil\frac{n_{i}}{3}\rceil-(r-1), & \text{if $r\geq1$;}\\
\sum\limits_{i=1}^{s}|D\cap V(P_{n_i})|+\sum\limits_{i=s+1}^{l}|D\cap V(P_{n_i}-x_{1}^{i})|\geq\sum\limits_{i=1}^{l}\lceil\frac{n_{i}}{3}\rceil, & \text{if $r=0$ and $s\geq1$;}\\
\sum\limits_{i=1}^{l}|D\cap V(P_{n_i}-x_{1}^{i})|+|D\cap N_{S}[c]|\geq\sum\limits_{i=1}^{l}\lceil\frac{n_{i}}{3}\rceil+1, & \text{if $r=0$ and $s=0$.}
\end{cases}
\end{equation*}

To prove the converse of above inequality, we need to construct a dominating set of $S$.
Set
\begin{equation*}
D_{i}=
\begin{cases}
\{x_{3}^{i},x_{6}^{i},\ldots,x_{n_{i-1}}^{i}\}, & \text{if $n_{i}\equiv1$(mod 3);}\\
\{x_{1}^{i},x_{4}^{i},\ldots,x_{n_{i-1}}^{i}\}, & \text{if $n_{i}\equiv2$(mod 3);}\\
\{x_{2}^{i},x_{5}^{i},\ldots,x_{n_{i-1}}^{i}\}, & \text{if $n_{i}\equiv0$(mod 3),}
\end{cases}
\end{equation*}
$i=1,2,\cdots,l$, and
\begin{equation*}
D_{0}=
\begin{cases}
\{c\}\cup \bigcup\limits_{i=1}^{l}D_{i}, & \text{if $r\geq1$;}\\
\bigcup\limits_{i=1}^{l}D_{i}, & \text{if $r=0$ and $s\geq1$;}\\
\{c\}\cup \bigcup\limits_{i=1}^{l}D_{i}, & \text{if $r=0$ and $s=0$.}
\end{cases}
\end{equation*}
Then $D_{0}$ is a dominating set of $S$. Hence
\begin{equation*}
\gamma(S)\leq|D_{0}|=
\begin{cases}
\sum\limits_{i=1}^{l}\lceil\frac{n_{i}}{3}\rceil-(r-1), & \text{if $r\geq1$;}\\
\sum\limits_{i=1}^{l}\lceil\frac{n_{i}}{3}\rceil, & \text{if $r=0$ and $s\geq1$;}\\
\sum\limits_{i=1}^{l}\lceil\frac{n_{i}}{3}\rceil+1, & \text{if $r=0$ and $s=0$.}
\end{cases}
\end{equation*}
\end{proof}

\begin{lemma}\label{bondagestar}
$b(K_{m}\boxtimes K_{1,n})\geq\lceil\frac{m}{2}\rceil$ for any $n\geq0$.
\end{lemma}
\begin{proof}
It is immediate from Lemma \ref{bondage star}.
\end{proof}

\begin{thm}
Let $S(n_{1},n_{2},\ldots,n_{l})$ be a starlike tree with at least two branches. Then
\begin{equation*}
b(K_{m}\boxtimes S)=
\begin{cases}
\lceil\frac{m}{2}\rceil, & \text{if $n_{1}\equiv n_{2}\equiv\cdots\equiv n_{l}\equiv1$ \emph{(mod 3);}}\\
m, & \text{if $n_{1}\equiv n_{2}\equiv\cdots\equiv n_{l}\equiv2$ \emph{(mod 3);}}\\
\lceil\frac{3m}{2}\rceil, & \text{if $n_{1}\equiv n_{2}\equiv\cdots\equiv n_{l}\equiv0$ \emph{(mod 3)}.}
\end{cases}
\end{equation*}
\end{thm}
\begin{proof}
{\it Case 1.} $n_{1}\equiv n_{2}\equiv\cdots\equiv n_{l}\equiv1$ (mod 3).

First, we prove $b(K_{m}\boxtimes S)\geq \lceil\frac{m}{2}\rceil$. Let $Z\subseteq E(K_{m}\boxtimes S)$ with $|Z|<\lceil\frac{m}{2}\rceil$. We need to prove that $\gamma(K_{m}\boxtimes S-Z)=\gamma(K_{m}\boxtimes S)$. Let $D_{c}\in\underline{MDS}(K_{m}\boxtimes S[N_{S}[c]]-Z)$ and $D_{i}\in\underline{MDS}(K_{m}\boxtimes (P_{n_i}-x_{1}^{i})-Z)$, $i=1,2,\ldots,l$. By Lemmas \ref{bondagestar} and \ref{lower-3}, we have $|D_{c}|=1$ and $|D_{i}|=\lceil\frac{n_{i}-1}{3}\rceil$ for every $1\leq i\leq l$. Thus by Proposition \ref{strong-p2} and Theorem \ref{sl-domi}, we have $\gamma(K_{m}\boxtimes S-Z)\leq |D_{c}\cup \bigcup\limits_{i=1}^{l}D_{i}|=1+\sum\limits_{i=1}^{l}\lceil\frac{n_{i}-1}{3}\rceil=\sum\limits_{i=1}^{l}\lceil\frac{n_{i}}{3}\rceil-(l-1)=\gamma(K_{m}\boxtimes S)$, which implies that $\gamma(K_{m}\boxtimes S-Z)=\gamma(K_{m}\boxtimes S)$.

Next, we prove $b(K_{m}\boxtimes S)\leq \lceil\frac{m}{2}\rceil$. It suffices to prove that $\gamma(K_{m}\boxtimes S-Z^{-}_{c})>\gamma(K_{m}\boxtimes S)$. (Recall that the definition of $Z^{-}_{c}$ was given in Subsection 3.2.)
Suppose to the contrary that $\gamma(K_{m}\boxtimes S-Z^{-}_{c})=\gamma(K_{m}\boxtimes S)$. So, for $D\in \gamma(K_{m}\boxtimes S-Z^{-}_{c})$, we have $D\in \gamma(K_{m}\boxtimes S)$.

\smallskip
{\it Claim 1.1.} $|D\cap V(K_{m}\boxtimes \{c\})|=1$.

Otherwise, we have $D\cap V(K_{m}\boxtimes \{c\}) =\emptyset$ by Lemma \ref{v}, which implies that $D\cap V(K_{m}\boxtimes P_{n_i})$ is a dominating set of $K_{m}\boxtimes P_{n_i}$ for every $1\leq i\leq l$. Hence $|D|=\sum\limits_{i=1}^{l}|D\cap V(K_{m}\boxtimes P_{n_i})|\geq\sum\limits_{i=1}^{l}\gamma(K_{m}\boxtimes P_{n_i})=\sum\limits_{i=1}^{l}\lceil\frac{n_{i}}{3}\rceil>\sum\limits_{i=1}^{l}\lceil\frac{n_{i}}{3}\rceil-(l-1)=|D|$, a contradiction.

\smallskip
{\it Claim 1.2.} $D\cap V(K_{m}\boxtimes S[N_{S}(c)])=\emptyset$.

It is similar to Lemma \ref{block-d}, we can easily deduce that $|D\cap V(K_{m}\boxtimes (P_{n_i}-x_{1}^{i}))|\geq \gamma(K_{m}\boxtimes(P_{n_i}-x_{1}^{i}-x_{2}^{i}))=\lceil\frac{n_{i}-2}{3}\rceil$ for every $1\leq i\leq l$. Hence
\begin{center}
$|D|  =  |D\cap V(K_{m}\boxtimes \{c\})|+|D\cap V(K_{m}\boxtimes S[N_{S}(c)])|+\sum\limits_{i=1}^{l}|D\cap V(K_{m}\boxtimes (P_{n_i}-x_{1}^{i}))|$ \\
$\geq 1+|D\cap V(K_{m}\boxtimes S[N_{S}(c)])|+\sum\limits_{i=1}^{l}\lceil\frac{n_{i}-2}{3}\rceil$ \ \ \ \ \ \ \ \ \ \ \ \ \ \ \ \ \ \ \ \ \ \ \ \ \ \ \ \ \ \ \ \ \ \ \ \ \ \ \ \ \ \ \ \ \ \ \\
$=  |D\cap V(K_{m}\boxtimes S[N_{S}(c)])|+\sum\limits_{i=1}^{l}\lceil\frac{n_{i}}{3}\rceil-(l-1)$  \ \ \ \ \ \ \ \ \ \ \ \ \ \ \ \ \ \ \ \ \ \ \ \ \ \ \ \ \ \ \ \ \ \ \ \ \ \ \ \ \ \\
 $ =  |D\cap V(K_{m}\boxtimes S[N_{S}(c)])|+|D|,$  \ \ \ \ \ \ \ \ \ \ \ \ \ \ \ \ \ \ \ \ \ \ \ \ \ \ \ \ \ \ \ \ \ \ \ \ \ \ \ \ \ \ \ \ \ \ \ \ \ \ \ \ \ \ \ \ \
\end{center}
which implies that $|D\cap V(K_{m}\boxtimes S[N_{S}(c)])|=0$.

By Claims 1.2, we can see that $D\cap V(K_{m}\boxtimes \{c\})$ is a dominating set of $K_{m}\boxtimes \{c\}-Z^{-}_{c}$, which implies that $\gamma(K_{m}\boxtimes \{c\}-Z^{-}_{c})=1=\gamma(K_{m}\boxtimes \{c\})$. But it is impossible since $Z^{-}_{c}$ is a bondage edge set of $K_{m}\boxtimes \{c\}$.

\smallskip
{\it Case 2.} $n_{1}\equiv n_{2}\equiv\cdots\equiv n_{l}\equiv2$ (mod 3).

First, we prove $b(K_{m}\boxtimes S)\geq m$. Let $Z\subseteq E(K_{m}\boxtimes S)$ with $|Z|<m$. It suffices to prove that $\gamma(K_{m}\boxtimes S-Z)=\gamma(K_{m}\boxtimes S)$. Let $D_{1,2}=\underline{MDS}(K_{m}\boxtimes S[V(\widetilde{P}_{n_1})\cup V(P_{n_2} )]-Z)$ and $D_{i}=\underline{MDS}(K_{m}\boxtimes P_{n_i}-Z)$, $i=3,4,\ldots,l$. By Lemma \ref{lower-2}, we have $|D_{1,2}|=\lceil\frac{n_{1}+1+n_{2}}{3}\rceil=\lceil\frac{n_{1}}{3}\rceil+\lceil\frac{n_{2}}{3}\rceil$ and $|D_{i}|=\lceil\frac{n_{i}}{3}\rceil$ for every $3\leq i\leq l$. Hence $\gamma(K_{m}\boxtimes S-Z)\leq|D_{1,2}|+\sum\limits_{i=3}^{l}|D_{i}|=\sum\limits_{i=1}^{l}\lceil\frac{n_{i}}{3}\rceil=\gamma(K_{m}\boxtimes S)$, which implies that $\gamma(K_{m}\boxtimes S-Z)=\gamma(K_{m}\boxtimes S)$.

 Next, we prove $b(K_{m}\boxtimes S)\leq m$. Recall that we have set $L_{n_{1}}=x^{1}_{1}x^{1}_{2}\cdots x^{1}_{n_{1}-1}x^{1}_{n_{1}}$ with $x^{1}_{n_{1}}$ being a vertex of degree one in $S$. Since $|Z_{x^{1}_{n_{1}-1}x^{1}_{n_{1}}}^{\thinspace\shortmid}|=m$, we just need to prove that $\gamma(K_{m}\boxtimes S-Z_{x^{1}_{n_{1}-1}x^{1}_{n_{1}}}^{\thinspace\shortmid})> \gamma(K_{m}\boxtimes S)$. Suppose to the contrary that $\gamma(K_{m}\boxtimes S-Z_{x^{1}_{n_{1}-1}x^{1}_{n_{1}}}^{\thinspace\shortmid})= \gamma(K_{m}\boxtimes S)$. Then, for $D\in \underline{MDS}(K_{m}\boxtimes S-Z_{x^{1}_{n_{1}-1}x^{1}_{n_{1}}}^{\thinspace\shortmid})$, we have $D\in \underline{MDS}(K_{m}\boxtimes S)$.

\smallskip
{\it Claim 2.1.} $D\cap V(K_{m}\boxtimes \{c\})=\emptyset$.

It is similar to Lemma \ref{block-d}, we can deduce that $|D\cap V(K_{m}\boxtimes P_{n_i})|\geq \gamma(K_{m}\boxtimes (P_{n_i}-x_{1}^{i}))=\lceil\frac{n_{i}-1}{3}\rceil=\lceil\frac{n_{i}}{3}\rceil$ for every $1\leq i\leq l$. Hence $|D|\geq |D\cap V(K_{m}\boxtimes \{c\}) |+\sum\limits_{i=1}^{l}|D\cap V(K_{m}\boxtimes P_{n_i})|\geq |D\cap V(K_{m}\boxtimes \{c\}) |+\sum\limits_{i=1}^{l}\lceil\frac{n_{i}}{3}\rceil=|D\cap V(K_{m}\boxtimes \{c\})|+|D|$, which implies that $|D\cap V(K_{m}\boxtimes \{c\})|=0$.

\smallskip
{\it Claim 2.2.} $D\cap V(K_{m}\boxtimes P_{n_1})\in\underline{MDS}(K_{m}\boxtimes P_{n_1})$.

By Claim 2.1, we have $D\cap V(K_{m}\boxtimes P_{n_i})$ is a dominating set of $K_{m}\boxtimes P_{n_i}$ for every $1\leq i\leq l$. So we get that $|D|\geq\sum\limits_{i=1}^{l}|D\cap V(K_{m}\boxtimes P_{n_i})|\geq\sum\limits_{i=1}^{l}\gamma(K_{m}\boxtimes L_{n_i})=\sum\limits_{i=1}^{l}\lceil\frac{n_{i}}{3}\rceil=|D|$, which implies that $|D\cap V(K_{m}\boxtimes P_{n_i})|=\gamma(K_{m}\boxtimes P_{n_i})$.
Thus, $D\cap V(K_{m}\boxtimes P_{n_i})\in\underline{MDS}(K_{m}\boxtimes P_{n_i})$ for every $1\leq i\leq l$.

\smallskip
Now, if $n_1>2$, then we have $|D\cap V(K_{m}\boxtimes \{x_{n_{1}-2}^{1}\})|=0$ by Claim 2.2 and Lemma \ref{block-0} (b); if $n_1=2$, then we have $|D\cap V(K_{m}\boxtimes \{c\})|=0$ by Claim 2.1. From these two observations, we can see that $D\cap V(K_{m}\boxtimes \{x_{n_{1}-1}^{1}x_{n_{1}}^{1}\})$ is a dominating set of $K_{m}\boxtimes \{x_{n_{1}-1}^{1}x_{n_{1}}^{1}\}-Z_{x^{1}_{n_{1}-1}x^{1}_{n_{1}}}^{\thinspace\shortmid}$. But it is impossible since we have $\gamma(K_{m}\boxtimes \{x_{n_{1}-1}^{1}x_{n_{1}}^{1}\}-Z_{x^{1}_{n_{1}-1}x^{1}_{n_{1}}}^{\thinspace\shortmid})>1=|D\cap V(K_{m}\boxtimes \{x_{n_{1}-1}^{1}x_{n_{1}}^{1}\})|$ by Lemmas \ref{vxy} and \ref{st}.
Hence $b(K_{m}\boxtimes S)\leq m$.

\smallskip
{\it Case 3.} $k_{1}\equiv k_{2}\equiv\cdots\equiv k_{l}\equiv0$ (mod 3).

First, we prove $b(K_{m}\boxtimes S)\geq \lceil\frac{3m}{2}\rceil$. Let $Z\in E(K_{m}\boxtimes S)$ with $|Z|<\lceil\frac{3m}{2}\rceil$. Assume without loss of generality that

\

\begin{center}
$|E(K_{m}\boxtimes P_{n_1})\cap Z|=$max$\{|E(K_{m}\boxtimes P_{n_i})\cap Z|:1\leq i\leq l\}$; \\

$|E(K_{m}\boxtimes P_{n_2})\cap Z|=$max$\{|E(K_{m}\boxtimes P_{n_i})\cap Z|:2\leq i\leq l\}$.
\end{center}

\noindent Then we must have $|E(K_{m}\boxtimes P_{n_i})\cap Z|<\lceil\frac{m}{2}\rceil$ for every $3\leq i\leq l$. Let $D_{1,2}\in\underline{MDS}(K_{m}\boxtimes S[V(\widetilde{P}_{n_1})\cup V(P_{n_2})]-Z)$ and $D_{i}\in\underline{MDS}(K_{m}\boxtimes P_{n_i}-Z)$, $i=3,4,\ldots,l$. By Lemmas \ref{lower-1} and \ref{lower-3}, we have $|D_{1,2}|=\lceil\frac{n_{1}+1+n_{2}}{3}\rceil$ and $|D_{i}|=\lceil\frac{n_{i}}{3}\rceil$ for every $3\leq i\leq l$. Thus $\gamma(K_{m}\boxtimes S-Z)\leq|D_{1,2}|+\sum\limits_{i=3}^{l}|D_{i}|=\sum\limits_{i=1}^{l}\lceil\frac{n_{i}}{3}\rceil+1=\gamma(K_{m}\boxtimes S)$, which implies that $\gamma(K_{m}\boxtimes S-Z)=\gamma(K_{m}\boxtimes S)$. So $b(K_{m}\boxtimes S)\geq \lceil\frac{3m}{2}\rceil$.

On the other hand, we have $b(K_{m}\boxtimes S)\leq \lceil\frac{3m}{2}\rceil$ by Theorem \ref{suspended vertex}. Hence $b(K_{m}\boxtimes S)= \lceil\frac{3m}{2}\rceil$.
\end{proof}

\end{document}